\documentclass{amsart}

\usepackage[english]{babel}
\usepackage[utf8x]{inputenc}
\usepackage[T1]{fontenc}
\usepackage{comment}
\usepackage{siunitx}

\usepackage{amsmath}
\usepackage{amssymb}
\usepackage{amsthm}
\usepackage{graphicx}
\usepackage[colorinlistoftodos]{todonotes}
\usepackage[colorlinks=true, allcolors=blue]{hyperref}
\usepackage{url}
\usepackage{xcolor}
\usepackage{tikz}
\usepackage{pgfplots}
\usepackage{tikz-cd}
\usepackage{algpseudocode}
\usepackage{tkz-euclide}

\usetikzlibrary{positioning, fit, arrows, shapes.geometric, backgrounds, calc}

\newcommand{\NN}{\mathbb{N}}

\newcommand{\RR}{\mathbb{R}}

\newcommand{\txt}[1]{\textnormal{#1}}
\newcommand{\op}{\txt{op}}

\newtheorem{thm}{Theorem}
\theoremstyle{definition}
\newtheorem{theorem}[thm]{Theorem}
\newtheorem{lemma}[thm]{Lemma}
\newtheorem{definition}[thm]{Definition}

\newtheorem{remark}[thm]{Remark}

\newtheorem{corollary}[thm]{Corollary}
\newtheorem{conjecture}{Conjecture}

\numberwithin{thm}{section}

\title[On the Erd\H{o}s-Tuza-Valtr Conjecture]{On the Erd\H{o}s-Tuza-Valtr Conjecture}

\author{Jineon Baek}
\address{Jineon Baek, University of Michigan, Department of Mathematics,  
2074 East Hall,
530 Church Street,
Ann Arbor, MI 48109-1043}
\email{jineon@umich.edu}

\begin{document}
\maketitle
\sloppy

\begin{abstract}

The Erd\H{o}s-Szekeres conjecture states that any set of more than $2^{n-2}$ points in the plane with no three on a line contains the vertices of a convex $n$-gon. 
Erd\H{o}s, Tuza, and Valtr strengthened the conjecture by stating that any set of more than $\sum_{i = n - b}^{a - 2} \binom{n - 2}{i}$ points in a plane either contains the vertices of a convex $n$-gon, $a$ points lying on a concave downward curve, or $b$ points lying on a concave upward curve.
They also showed that the generalization is actually equivalent to the Erd\H{o}s-Szekeres conjecture.

We prove the first new case of the Erd\H{o}s-Tuza-Valtr conjecture since the original 1935 paper of Erd\H{o}s and Szekeres. 
Namely, we show that any set of $\binom{n-1}{2} + 2$ points in the plane with no three points on a line and no two points sharing the same $x$-coordinate either contains 4 points lying on a concave downward curve or the vertices of a convex $n$-gon.
\end{abstract}

\section{Introduction}
A set of $n$ points on a plane is in a \emph{convex position} if they form the vertices of a convex polygon.
For simplicity, denote any set of $n$ points in convex position as an \emph{$n$-gon}.
The well-known Erd\H{o}s-Szekeres conjecture is the following.
Here, a finite set of points on a plane is in \emph{general position} if no three points are on a line and no two points share the same $x$-coordinate\footnote{It is not usually required for the $x$-coordinates to be different when stating the Erd\H{o}s-Szekeres conjecture. However, the extra assumption does not hurt generality as we can rotate the point set slightly if there is any overlap in the $x$-coordinates.}.

\begin{definition}
For any $n \geq 3$, let $N(n)$ be the maximum number of points on a plane in a general position with no subset forming an $n$-gon.
\end{definition}
\begin{conjecture}[Erd\H{o}s-Szekeres \cite{erd1960some}]
\label{conj:es}
For any $n \geq 3$, $N(n) = 2^{n-2}$.
\end{conjecture}
 
In 1935, Erd\H{o}s and Szekeres showed $N(n) \leq \binom{2n-4}{n-2} = O\left(4^n/\sqrt{n}\right)$ \cite{erdos1935combinatorial}.
In 1960, they provided a construction of exactly $2^{n-2}$ points with no $n$ points in convex position, showing the lower bound $N(n) \geq 2^{n-2}$ \cite{erd1960some}.
The upper bound of $N(n)$ stayed in the magnitude of $O(4^n/\sqrt{n})$ despite many improvements 
\cite{chung1998forced,kleitman1998finding,toth1998note,toth2005erdos, mojarrad2016improved,norin2016erdHos} until Suk \cite{suk2017erdHos} proved an upper bound of $N(n) \leq 2^{n + o(n)}$ in 2016.
The best upper bound so far is $N(n) \leq 2^{n + O(\sqrt{n \log n})}$ \cite{holmsen2020two}, and the precise equality $N(n) = 2^{n - 2}$ remains neither proven or disproven to this date.

Now we state the generalization of Conjecture \ref{conj:es} by Erd\H{o}s, Tuza, and Valtr \cite{erdos1996ramsey}.

\begin{definition}
An \emph{$a$-cap} (resp. \emph{$a$-cup}) is a set of $a$ points lying on the graph of a downwardly (resp. upwardly) convex function (see Figure \ref{fig:cap-cup-gon}).
\end{definition}

\begin{definition}
Call any tuple $(n, a, b)$ of integers
satisfying $2 \leq a, b \leq n \leq a + b - 2$ a \emph{triplet}.
For any triplet $(n, a, b)$, 
define $N(n, a, b)$ as the maximum number of points on a plane in general position with no subset forming an $n$-gon, $a$-cap, or $b$-cup.
\end{definition}

\colorlet{lightblue}{blue!15}
\begin{center}
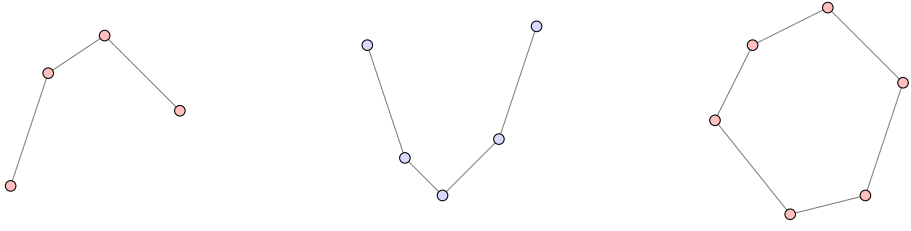
\begin{figure}[t]
\minipage{0.25\textwidth}
\centering 
\begin{tikzpicture}[scale=0.5]
  \draw[gray!100] (-1, 0) -- (0, 3) -- (1.5, 4) -- (3.5, 2);
  \draw[fill=pink] (-1, 0) circle (4pt);
  \draw[fill=pink] (0, 3) circle (4pt);
  \draw[fill=pink] (1.5, 4) circle (4pt);
  \draw[fill=pink] (3.5, 2) circle (4pt);
\end{tikzpicture}
\endminipage\hfill
\minipage{0.25\textwidth}
\centering 
\begin{tikzpicture}[scale=0.5]
  \draw[gray!100] (0, 3) -- (1, 0) -- (2, -1) -- (3.5, 0.5) -- (4.5, 3.5);
  \draw[fill=lightblue] (0, 3) circle (4pt);
  \draw[fill=lightblue] (1, 0) circle (4pt);
  \draw[fill=lightblue] (2, -1) circle (4pt);
  \draw[fill=lightblue] (3.5, 0.5) circle (4pt);
  \draw[fill=lightblue] (4.5, 3.5) circle (4pt);
\end{tikzpicture}
\endminipage\hfill
\minipage{0.25\textwidth}%
\centering 
\begin{tikzpicture}[scale=0.5]
  \draw[gray!100] (-1, 1) -- (0, 3) -- (2, 4) -- (4, 2) -- (3, -1) -- (1, -1.5) -- (-1, 1);
  \draw[fill=pink] (-1, 1) circle (4pt);
  \draw[fill=pink] (0, 3) circle (4pt);
  \draw[fill=pink] (2, 4) circle (4pt);
  \draw[fill=pink] (4, 2) circle (4pt);
  \draw[fill=pink] (3, -1) circle (4pt);
  \draw[fill=pink] (1, -1.5) circle (4pt);
\end{tikzpicture}
\endminipage
\caption{A 4-cap, 5-cup and 6-gon (from left to right)}
\label{fig:cap-cup-gon}
\end{figure}
\end{center}

\begin{conjecture}[Erd\H{o}s-Tuza-Valtr \cite{erdos1996ramsey}]
\label{conj:etv}
For any triplet $(n, a, b)$, we have
$$N(n, a, b) = \sum_{i = n - b}^{a - 2} \binom{n - 2}{i}.$$
\end{conjecture}

\begin{definition}
For any triplet $(n, a, b)$, let $P(n, a, b)$ be the statement that $N(n, a, b) = \sum_{i = n - b}^{a - 2} \binom{n - 2}{i}.$
That is, $P(n, a, b)$ is the special case of Conjecture $\ref{conj:es}$ with the triplet $(n, a, b)$.
\end{definition}

Then the special case $P(n, n, n)$ of Conjecture \ref{conj:etv} is Conjecture \ref{conj:es}, as $n$-caps/cups are $n$-gons and the conjectured sum equals $2^{n-2}$.
Interestingly, they showed that the strengthened Conjecture \ref{conj:etv} is actually \emph{implied by} the original Conjecture \ref{conj:es} as well.

\begin{theorem}[\cite{erdos1996ramsey}]
\label{thm:etv}
Conjecture \ref{conj:es} and Conjecture \ref{conj:etv} are equivalent.
Specifically, for any $n \geq 3$ and any triplets $(n, a, b)$, $(n, a', b')$ with $a \geq a'$ and $b \geq b'$,
the statement $P(n, a, b)$ implies $P(n, a', b')$.
\end{theorem}

So any counterexample of Conjecture \ref{conj:etv} for one triplet $(n, a, b)$ would disprove Conjecture \ref{conj:es} immediately.
Indeed, such a line of attack \cite{balko2017sat} succeeded in disproving a set-theoretic generalization \cite{szekeres2006computer} of Conjecture \ref{conj:es}  by disproving an analogous set-theoretic generalization of Conjecture \ref{conj:etv} for the triplet $(n, 4, n)$.

It is then natural to ask whether the original Conjecture \ref{conj:etv} will hold for the same triplet $(n, 4, n)$ or not.
We show that it does.

\begin{theorem}[Main Theorem]
\label{thm:main}
	For any $n \geq 3$, the statement $P(n, 4, n)$ holds.
	That is, $\binom{n - 1}{2} + 2$ points on a plane in general position determine either a 4-cap or a $n$-gon.
\end{theorem}

Related, Erd\H{o}s and Szekeres proved the following \emph{cups-caps theorem} in 1935 \cite{erdos1935combinatorial}.
\begin{theorem}[Erd\H{o}s–Szekeres]
\label{thm:capcup}
For any $a, b \geq 2$, any set of more than $\binom{a+b-4}{a-2}$ points in general position contains either an $a$-cap or $b$-cup.
\end{theorem}
For any $a, b \geq 2$, the cases $P(a + b - 3, a, b)$ and $P(a + b - 2, a, b)$ of Conjecture \ref{conj:etv} are immediate consequences of this theorem. 
To see this, observe that any $(a + b - 3)$-gon either contains an $a$-cap on top or a $b$-cap on the bottom, and that the conjectured values $N(a + b - 3, a, b) = \binom{a+b-5}{a-3} + \binom{a+b-5}{a-2} = \binom{a+b-4}{a-2}$ and $N(a + b - 2, a, b) = \binom{a+b-4}{a-2}$ match with the cups-caps theorem.
Our Theorem \ref{thm:main} is not a consequence of the cups-caps theorem as $n < a + b - 3$ in the triplet $(n, 4, n)$.

By Theorem \ref{thm:etv}, for triplets $(n, a, b)$ with a fixed $n$, the cases $P(n, a, b)$ of Conjecture \ref{conj:etv} form a pyramid of implications from top to bottom (Figure \ref{fig:pyramid}).
On the top, we have the Erd\H{o}s-Szekeres conjecture (Conjecture \ref{conj:es}),
and at the bottom, we have the cups-caps theorem (Theorem \ref{thm:capcup}).
Our theorem lies strictly in the middle of them (Theorem \ref{thm:main}).
To the best of our knowledge, this is the first new instance of Conjecture \ref{conj:etv} proven so far since the original 1935 paper \cite{erdos1935combinatorial} of Erdős and Szekeres.

\begin{center}
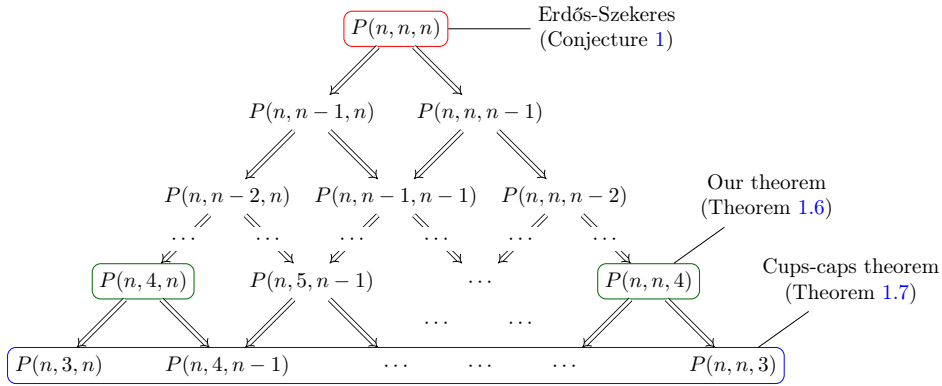
\begin{figure}[h]
\centering 
\scalebox{0.8}{
\begin{tikzpicture}
\begin{scope}[every node/.style={rectangle}]
	\def\ux{1.4}
	\def\uy{1.4}
    \node[draw=red, fill=white, rounded corners] (N00) at (0,0) {$P(n, n, n)$};
    \node (N10) at (-\ux,-1*\uy) {$P(n, n - 1, n)$};
    \node (N11) at (\ux,-1*\uy) {$P(n, n, n - 1)$};
    \node (N20) at (-2*\ux,-2*\uy) {$P(n, n - 2, n)$};
    \node (N21) at (0,-2*\uy) {$P(n, n - 1, n - 1)$};
    \node (N22) at (2*\ux,-2*\uy) {$P(n, n, n - 2)$};
    \node[draw=black!60!green, fill=none, rounded corners] (N30) at (-3*\ux,-3*\uy) {$P(n, 4, n)$};
    \node (N31) at (-1*\ux,-3*\uy) {$P(n, 5, n - 1)$};
    \node (N32) at (1*\ux,-3*\uy) {$\phantom{P(}\cdots\phantom{P)}$};
    \node[draw=black!60!green, fill=none, rounded corners] (N33) at (3*\ux,-3*\uy) {$P(n, n, 4)$};
    \node (N40) at (-4*\ux,-4*\uy) {$P(n, 3, n)$};
    \node (N41) at (-2*\ux,-4*\uy) {$P(n, 4, n - 1)$};
    \node (N42) at (0*\ux,-4*\uy) {$\phantom{P(}\cdots\phantom{P)}$};
    \node (N42.5) at (\ux,-4*\uy) {$\phantom{P(}\cdots\phantom{P)}$};
    \node (N43) at (2*\ux,-4*\uy) {$\phantom{P(}\cdots\phantom{P)}$};
    \node (N44) at (4*\ux,-4*\uy) {$P(n, n, 3)$};
    
    \node[on background layer, draw=blue, fill=none, rectangle, rounded corners, fit=(N40) (N44), inner sep=0] (all) {};
    
    \node[align=center] (L0) at (2.5*\ux, 0) {Erd\H{o}s-Szekeres \\ (Conjecture \ref{conj:es})};
    \draw (N00) -- (L0);
    \node[align=center] (L1) at (4.4*\ux, -2*\uy) {Our theorem \\ (Theorem \ref{thm:main})};
    \draw (N33) -- (L1);
    \node[align=center] (L2) at (5.4*\ux, -3*\uy) {Cups-caps theorem \\ (Theorem \ref{thm:capcup})};
    \draw (N44) -- (L2);
\end{scope}

\begin{scope}
	\draw[-implies, double equal sign distance] (N00) -- (N10);
	\draw[-implies, double equal sign distance] (N00) -- (N11);
	\draw[-implies, double equal sign distance] (N10) -- (N20);
	\draw[-implies, double equal sign distance] (N10) -- (N21);
	\draw[-implies, double equal sign distance] (N11) -- (N21);
	\draw[-implies, double equal sign distance] (N11) -- (N22);
	
	\draw[-implies, double equal sign distance] (N20) -- (N30) node [midway, fill, color=white, text=black] {$\cdots$};
	\draw[-implies, double equal sign distance] (N20) -- (N31) node [midway, fill, color=white, text=black] {$\cdots$};
	\draw[-implies, double equal sign distance] (N21) -- (N31) node [midway, fill, color=white, text=black] {$\cdots$};
	\draw[-implies, double equal sign distance] (N21) -- (N32) node [midway, fill, color=white, text=black] {$\cdots$};
	\draw[-implies, double equal sign distance] (N22) -- (N32) node [midway, fill, color=white, text=black] {$\cdots$};
	\draw[-implies, double equal sign distance] (N22) -- (N33) node [midway, fill, color=white, text=black] {$\cdots$};
	
	\draw[-implies, double equal sign distance] (N30) -- (N40);
	\draw[-implies, double equal sign distance] (N30) -- (N41);
	\draw[-implies, double equal sign distance] (N31) -- (N41);
	\draw[-implies, double equal sign distance] (N31) -- (N42);
	\draw[draw=none, -implies, double equal sign distance] (N32) -- (N42) node [midway] {$\cdots$};
	\draw[draw=none, -implies, double equal sign distance] (N32) -- (N43) node [midway] {$\cdots$};
	\draw[-implies, double equal sign distance] (N33) -- (N43);
	\draw[-implies, double equal sign distance] (N33) -- (N44);
\end{scope}

\end{tikzpicture}
}
\caption{A schematic diagram of the implications between Conjecture \ref{conj:es}, Theorem \ref{thm:capcup} and our Theorem \ref{thm:main}.}
\label{fig:pyramid}
\end{figure}
\end{center}

The proof of main Theorem \ref{thm:main} generalizes to a purely combinatorial model of convexity (Theorem \ref{thm:main-thm}). 
Section \ref{sec:set} describes the combinatorial setup we mainly work with.
Section \ref{sec:statistic} describes the \emph{$\alpha$-statistic}, a function for understanding $a$-cap, $b$-cup free configurations especially with nearly maximal number of points.
Section \ref{sec:alpha-beta} introduces $(\alpha, \beta)$-plane, an useful notion for understanding 4-cap free configurations, 
and motivates the definition of \emph{interweaved laced cups} (Definition \ref{def:inter-laced} and \ref{def:laced}).
Finally, Section \ref{sec:laced} proves main Theorem \ref{thm:main} by finding interweaved laced cups with induction (Theorem \ref{thm:main-lemma}). 

\section{A combinatorial model of convexity}
\label{sec:set}
We introduce a purely combinatorial model of convexity that we mainly work with. It is a slight modification of the one first suggested in \cite{szekeres2006computer} and explored further in \cite{fox2012erdHos}, \cite{moshkovitz2014ramsey} and \cite{balko2017sat}.

\begin{definition}
In this paper, a \emph{configuration} is a finite set $S$ of elements called \emph{points} or \emph{vertices} equipped with the following structures.
\begin{itemize}
	\item A linear ordering $<$ of the vertices.
	\item For any subset of $S$ with size 3, an arbitrary assignment of whether they form a cap or a cup.
\end{itemize}
In standard terms, a configuration is a bi-coloring of a finite, complete 3-uniform hypergraph with a prescribed linear ordering of vertices.
\end{definition}

Given a finite set of points in general position, we can make a configuration by ordering the points in their increasing $x$-coordinates and assigning each size 3 subset to be either a cap or cup according to its position in the plane.
Let's say that a configuration is \emph{realizable} if it can be constructed from actual points in this way.
The notion of caps and cups in plane are generalized to arbitrary configurations which may not be realizable.
\begin{definition}
Denote any set of vertices $x_1 < x_2 < \cdots < x_a$ of a configuration as simply $x_1x_2\cdots x_a$ in the increasing order. 

In a configuration, a set $C = x_1x_2 \cdots x_a$ of $a$ vertices forms an $a$-cup (resp. $a$-cap) if any three consecutive points $x_{i-1} x_i x_{i+1}$ form a cup (resp. cap)  for any $1 < i < a$.
In particular, we allow 1-cups and 1-caps.
The \emph{size} of $C$, also denoted $|C|$, is the number of vertices $a$ in $C$.

We say that $x_1$ is the \emph{starting point} of $C$ and $x_a$ is the \emph{ending point} of $C$.
The points $x_1$ and $x_a$ are the \emph{endpoints} of $C$.
Equivalently, we say that $C$ starts with $x_1$ and ends with $x_a$, or that $C$ is a cap (resp. cup) from $x_1$ to $x_a$.
Call any pair $x < y$ of vertices in a configuration an \emph{edge}.
If the size $a$ of $C$ is at least 2, we say that $C$ \emph{starts} with the edge $x_1x_2$ and \emph{ends} with the edge $x_{a-1}x_a$.

For a cap (resp. cup) $C$ from vertex $s$ to vertex $t$, say that $C$ \emph{extends to left} with the vertex $x$ (or edge $xs$) if $xC$ is also a cap (resp. cup).
Likewise, say that $C$ \emph{extends to right} with the vertex $x$ (or edge $tx$) if $Cx$ forms a cap (resp. cup).
\end{definition}

We also introduce the \emph{mirror reflection} of a configuration.
For realizable configurations, this corresponds to reflecting the points along the $y$-axis.
\begin{definition}
The \emph{mirror reflection} $S^{op}$ of a configuration $S$ 
is the configuration with the same vertex set and assignments of 3-caps and 3-cups, but the prescribed linear ordering reversed.
In this way, a cap (resp. cup) $C$ in $S$ naturally corresponds to the reflected cap (resp. cup) $C^{op}$ in $S^{op}$ with the same set of vertices.
\end{definition}

We introduce two possible generalizations of an $n$-gon in this combinatorial model of convexity.
\begin{definition}
In a configuration, a \emph{weak $(a, b)$-gon} is a pair of $a$-cap $C_1$ and $b$-cup $C_2$ sharing the same endpoints $x$ and $y$.
An \emph{weak $n$-gon} is any weak $(a, b)$-gon with $a + b = n + 2$.

On the other hand, a \emph{strong $(a, b)$-gon} is a weak $(a, b)$-gon $(C_1, C_2)$ with the additional constraint that $C_1 \cap C_2 = \{x, y\}$.
An \emph{strong $n$-gon} is any strong $(a, b)$-gon with $a + b = n + 2$.

An \emph{$(a, b)$-gon} or \emph{$n$-gon} denotes the weak $(a, b)$-gon or $n$-gon by default.
\end{definition}
The only difference is that a weak $n$-gon allows the cap and cup to have overlapping vertices in the middle, while a strong $n$-gon does not. 
Note also that there is no difference in any realizable configurations. 
We use the following notion.
\begin{definition}
A configuration is $a$-cap (resp. $b$-cup or weak/strong $n$-gon) free if it has no $a$-cap (resp. $b$-cup or weak/strong $n$-gon) in the configuration.	
\end{definition}
Peters and Szekeres \cite{szekeres2006computer} proposed to generalize Conjecture \ref{conj:es} by stating that for any $n \geq 2$, the maximum size of a \emph{strong} $n$-gon free configuration is $2^{n-2}$, and supplied a computer proof for the case $n = 6$.
Later, Balko and Valtr \cite{balko2017sat} found counterexamples for $n = 7$ and $8$ with SAT solvers by finding counterexamples for analogues of Conjecture \ref{conj:etv}.
We instead propose to use the definition of a \emph{weak} $n$-gon instead to generalize Conjecture \ref{conj:es}.
\begin{conjecture}[Set-theoretic Erd\H{o}s-Szekeres]
\label{conj:es-set}
For any $n \geq 2$, 
the maximum size $\widehat{N}(n)$ of a \emph{weak} $n$-gon free configuration is equal to 
$2^{n-2}$.
\end{conjecture}
From now on, we omit the word \emph{weak} when we mention any $(a, b)$-gon or $n$-gon.
We can also state the analogous generalization of Conjecture \ref{conj:etv} to arbitrary configurations. 
\begin{conjecture}[Set-theoretic Erd\H{o}s-Tuza-Valtr]
\label{conj:main-conj}
For any triplet $(n, a, b)$, the maximum size $\widehat{N}(n, a, b)$ of a weak $n$-gon, $a$-cap and $b$-cup free configuration is equal to $\sum_{i = n + 2 - b}^{a} \binom{n - 2}{i - 2}$.
\end{conjecture}

The generalization of cups-caps theorem (Theorem \ref{thm:capcup}) to arbitrary configurations is well-known (e.g. \cite{moshkovitz2014ramsey}).
\begin{theorem}[Set-theoretic cups-caps]
\label{thm:capcup-set}
For any $a, b \geq 2$, the maximum size of an $a$-cap and $b$-cup free configuration is $\binom{a+b-4}{a-2}$.
\end{theorem}

Our main result is the proof of Conjecture \ref{conj:main-conj}  for the case $(n, a, b) = (n, 4, n)$.
It is the generalization of Theorem \ref{thm:main} to arbitrary configurations.
\begin{theorem}
\label{thm:main-thm}
For any $n \geq 3$, any configuration of size $\binom{n-1}{2}+2$ contains either a 4-cap, $n$-cup or a $(3, n-1)$-gon.
\end{theorem}

\section{Structure of $a$-cap, $b$-cup free configurations}
\label{sec:statistic}

In this section, we discuss some properties of a configuration $S$ that avoids $a$-caps and $b$-cups.
They will be useful especially when the size of $S$ is close to the maximum possible value $\binom{a+b-4}{a-2}$ shown in Theorem \ref{thm:capcup-set}. 
We do not claim originality of the results in this section - it is mostly a recasting of the definitions introduced in \cite{moshkovitz2014ramsey}.

First, we define the \emph{slope labeling} of an $a$-cap free configuration $S$.
\begin{definition}
\label{def:label}
A \emph{slope labeling} $s$ of an $a$-cap free configuration $S$ is the assignment of an integer $s(xy) \in \{1, 2, \cdots, a - 2\}$ to all the edges $xy$ of $S$ so that it satisfies the following.
\begin{itemize}
	\item For any points $x < y < z$ in $S$, $s(xy) \leq s(yz)$ implies that $xyz$ is a 3-cup.
\end{itemize}
The value $s(xy)$ assigned to the edge $xy$ is the \emph{label} of $xy$.
\end{definition}

For the actual slope $s_\RR(xy) \in \RR$ of an edge in a realizable configuration $S$, both the following properties would hold.
\begin{itemize}
	\item For any points $x < y < z$ in $S$, $s_\RR(xy) > s_\RR(yz)$ implies that $xyz$ is a 3-cap.
	\item For any points $x < y < z$ in $S$, $s_\RR(xy) < s_\RR(yz)$ implies that $xyz$ is a 3-cup.
\end{itemize}

A slope labeling restricts the possible values of a `slope' (label) to a much smaller set $\{1, 2, \cdots, a - 2\}$ at the cost of giving up the first property. That is, edges of strictly decreasing labels may not form a cap (see the cup ABF in Figure \ref{fig:example}).
It is also impossible in general to assign a slope labeling that satisfies both the properties.
However, any $a$-cap free configuration has a slope labeling that satisfies the second property.

\begin{theorem}
\label{thm:label}
For any $a$-cap free configuration $S$, a slope labeling always exists.
In particular, for any edge $e$, let $c(e)$ be the maximum length of a cap starting with $e$; the function $c(e)-1$ is a slope labeling.
\end{theorem}
\begin{proof}
For any edge $e$, its assigned label $i$ is in between 1 and $a-2$ inclusive by definition.
Assume by contrary that there is a 3-cap $xyz$ in $S$
with the label $s$ of $xy$ less than or equal to the label $t$ of $yz$.
As the label of edge $yz$ is $t$, there is a cap $C$ of size $t + 1$ starting with $yz$.
Since $xyz$ forms a cap, the cap $C$ extends to a cap $xC$ of size $t + 2$ that starts with $xy$.
By the definition of $s$, we have $s + 1 \geq t + 2$, which contradicts the hypothesis $s \leq t$.
\end{proof}

\begin{center}
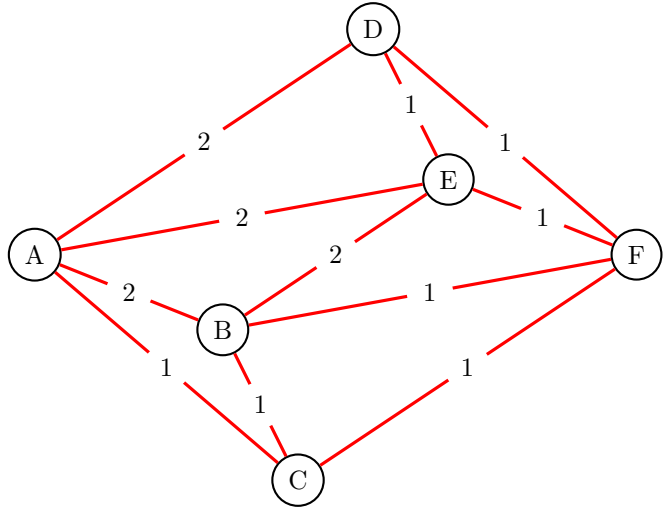
\begin{figure}[h]
\centering 
\begin{tikzpicture}
\begin{scope}[every node/.style={circle,thick,draw}]
    \node (A) at (-4,0) {A};
    \node (B) at (-1.5,-1) {B};
    \node (C) at (-0.5,-3) {C};
    \node (D) at (0.5,3) {D};
    \node (E) at (1.5,1) {E};
    \node (F) at (4,0) {F} ;
\end{scope}

\begin{scope}[>={Stealth[black]},
              every node/.style={fill=white,circle},
              every edge/.style={draw=red,very thick}]
    \path [-] (A) edge node {$2$} (B);
    \path [-] (A) edge node {$1$} (C);
    \path [-] (A) edge node {$2$} (D);
    \path [-] (A) edge node {$2$} (E);
    \path [-] (B) edge node {$1$} (C);
    \path [-] (B) edge node {$2$} (E);
    \path [-] (B) edge node {$1$} (F);
    \path [-] (C) edge node {$1$} (F);
    \path [-] (D) edge node {$1$} (E); 
    \path [-] (D) edge node {$1$} (F); 
    \path [-] (E) edge node {$1$} (F); 
\end{scope}
\end{tikzpicture}
\caption{
A $4$-cap and $4$-cup free realizable configuration with the slope labeling of Theorem \ref{thm:label}. 
The labels of the edges AF, BD, CD and CE are 1, 2, 2 and 2 respectively.
}
\label{fig:example}
\end{figure}
\end{center}

For the rest of the paper, we will fix a large $a$-cap free configuration $S$
and a slope labeling of $S$ so that there is no confusion in the label of an arbitrary edge of $S$.
The following lemma is immediate.
\begin{lemma}
For any $a$-cap free configuration $S$ with slope labeling $s$, the restriction of $s$ to an arbitrary subset $S'$ of $S$ is a slope labeling of $S'$.
\end{lemma}

Assume any $a$-cap, $b$-cup free configuration $S$ with a slope labeling $s$ by integers from 1 to $a-2$. 
For each point $p$ of $S$, we assign its \emph{$\alpha$-statistic} which is a tuple $\alpha(p) = (\alpha_1(p), \alpha_2(p), \cdots, \alpha_{a-2}(p))$ of natural numbers.

\begin{definition}
Fix any $a$-cap, $b$-cup free configuration $S$ with a slope labeling.
For all $p \in S$ and $1 \leq i \leq a - 2$, define the integer $\alpha_i(p)$ as maximum length of a cup that ends with the point $p$ and an edge of slope $\leq i$.
If there is no such cup, then let $\alpha_i(p) = 1$.
In particular, the rightmost point $p$ has the value $\alpha^i(p) = 1$ for any $i$.
The \emph{$\alpha$-statistic} of a point $p$ in $S$ is 
the tuple $\alpha(p) = (\alpha_1(p), \alpha_2(p), \cdots, \alpha_{a-2}(p))$.
\end{definition}

\begin{definition}
For any integers $a, b \geq 2$, define the grid simplex
$$T_{a, b} := \{(x_1, x_2, \cdots, x_{a-2}) \in \NN^{a - 2} \colon 1 \leq x_1 \leq \cdots \leq x_{a-2} \leq b - 1\}.$$
\end{definition}

\begin{theorem}
\label{thm:inj}
For any $a$-cap, $b$-cup free configuration $S$ with a fixed slope labeling $s$, the $\alpha$-statistic of $S$ satisfies the following.
\begin{enumerate}
	\item For any point $p$, $1 \leq \alpha_1(p) \leq \cdots \leq \alpha_{a-2} (p) \leq b - 1$. Consequently, the $\alpha$-statistic $\alpha$ is a map from $S$ to $T_{a, b}$.
	\item For any two points $x < y$ connected by an edge of label $i$, $\alpha_i(x) < \alpha_i(y)$ holds. So in particular, $\alpha$ is injective.
\end{enumerate}
\end{theorem}
\begin{proof}
We check the two conditions the $\alpha$-statistic has to satisfy.
The inequality $1 \leq \alpha_1(p) \leq \cdots \leq \alpha_{a-2} (p) \leq b - 1$ for any point $p$ follows directly from the definitions.
The cup of length $\alpha_i(x)$ that ends with the vertex $x$ and an edge of label $\leq i$ can be extended to the right with the edge $xy$ of label $i$, so that $\alpha_i(x) + 1 \leq \alpha_i(y)$ in any of the two given definitions of $\alpha_i$.
\end{proof}
 
Note that the size of $T_{a, b}$ is $\binom{a+b-4}{a-2}$.
For a quick proof, note that any $(x_1, x_2, \cdots, x_{a-2}) \in T_{a, b}$ corresponds bijectively
to an arbitrary subset $\{x_i + i - 1 \colon 1 \leq i \leq a-2\}$ of $\{1, 2, \cdots, a + b - 4 \}$ of size $a-2$.
Consequently, $\alpha$ being injective immediately proves Theorem \ref{thm:capcup-set} that $|S| \leq \binom{a+b-4}{a-2}$.
Therefore, if the size of $S$ is nearly equal to the maximum size $\binom{a+b-4}{a-2}$,
we can expect the map $\alpha$ to be almost bijective with some exceptional `holes' in $T_{a, b}$.

Let us end this section with a remark on the $\alpha$-statistic of the mirror configuration.
\begin{definition}
For any $a$-cap, $b$-cup free configuration $S$ and its slope labeling $s$, define the mirror reflection $s^\op$ of the slope labeling $s$ as the following.
$$s^{\op}(yx) = a - 1 -  s(xy)$$
\end{definition}
We leave it as an exercise to show that $s^\op$ is a proper slope labeling of $S^\op$.
\begin{lemma}
The mirror reflection $s^\op$ of a slope labeling $s$ of $S$ is a slope labeling of $S^\op$.
\end{lemma}
\begin{remark}
\label{rem:4cap-mirror}
We use the mirror reflection of a labeling extensively to exploit the symmetry without loss of generality. 
For example, take an arbitrary 4-cap free configuration $S$. Its slope labeling $s$ has values in $\{1, 2\}$. So for any edge $xy$, if its label $s(xy)$ is 1 (resp. 2) then the labeling $s^\op(yx)$ of its reflection is 2 (resp. 1).
Therefore, if the statement we want to show regarding $S$ is invariant under reflection,
then we can safely assume that the edge connecting $x$ and $y$ is labeled 1 by reflecting $S$ if necessary.
\end{remark}

\section{The $(\alpha, \beta)$-plane}
\label{sec:alpha-beta}
In this section, we focus our attention on an arbitrary 4-cap, $n$-cup free configuration $S$ with a fixed slope labeling and $\alpha$-statistic.

\begin{center}
\begin{figure}[h]
\centering 
\begin{tikzpicture}[scale=1]

\draw[step=2cm,gray] (-0.5,-0.5) grid (6.5,6.5);
\draw[-stealth,thick] (-0.5,0)--(6.5,0) node[below right]{$\alpha(p)$}; 
\draw[-stealth,thick] (0,-0.5)--(0,6.5) node[above left]{$\beta(p)$};
\foreach \i in {1,2,3}
\draw[very thick] 
(\i+\i,.1)--(\i+\i,-.1) node[below,fill=white]{\i}
(-.1,\i+\i) node[left,fill=white]{\i}--(.1,\i+\i) ;

\begin{scope}[every node/.style={circle,thick,draw,fill=white}]
    \node (A) at (2, 2) {A};
    \node (B) at (2, 4) {B};
    \node (C) at (4, 4) {C};
    \node (D) at (2, 6) {D};
    \node (E) at (4, 6) {E};
    \node (F) at (6, 6) {F};
\end{scope}

\begin{scope}[>={Stealth[black]},
              every node/.style={circle},
              every edge/.style={draw=gray, thick}]
    \path [-] (A) edge node {$2$} (B);
    \path [-] (A) edge node {$1$} (C);
    \path [-] (A) edge [bend left = 40] node {$2$} (D);
    \path [-] (B) edge node {$2$} (D);
    \path [-] (B) edge node {$1$} (C);
    \path [-] (B) edge node {$2$} (E);
    \path [-] (C) edge node {$2$} (E);
    \path [-] (C) edge node {$1$} (F);
    \path [-] (D) edge node {$1$} (E); 
    \path [-] (D) edge [bend left = 40] node {$1$} (F); 
    \path [-] (E) edge node {$1$} (F); 
\end{scope}

\end{tikzpicture}
\caption{
The $(\alpha, \beta)$-plane of the 4-cap, 4-cup free configuration in Figure \ref{fig:example} with  the slope labels of some edges.
}
\label{fig:example-stat}
\end{figure}
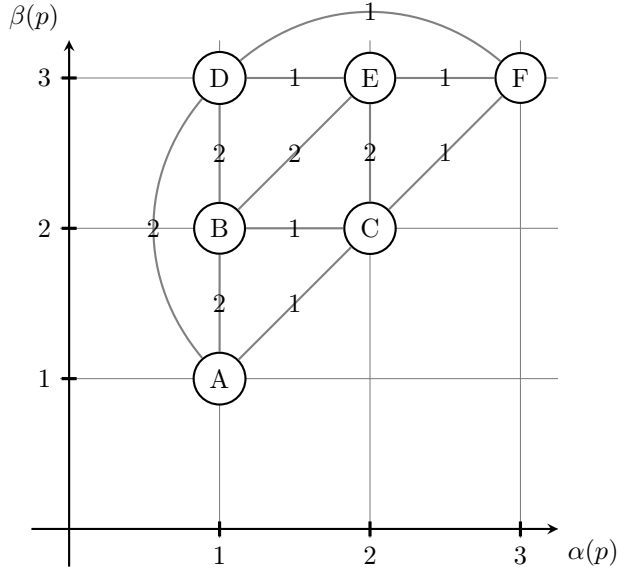
\end{center}

In our configuration $S$, we only have edges of label 1 and 2, and the following is an immediate consequence of Definition \ref{def:label}.
We will use this extensively without further mentions.
\begin{corollary}
In an arbitrary 4-cap free configuration with slope labeling, any edge of label 1 extends a cup to the left.
That is, for any edge $pq$ of label 1 and a cup $C$ that starts with $q$, the sequence $pC$ is also a cup.

Likewise, any edge of label 2 extends a cup to the right.
That is, for any edge $pq$ of label 2 and a cup $C$ that ends with $p$, the sequence $Cq$ is also a cup.
\end{corollary}

Also, we define aliases for the $\alpha$-statistic of $S$.
\begin{definition}
For an arbitrary 4-cap, $n$-cup free configuration $S$ with a fixed slope labeling and $\alpha$-statistic $p \mapsto (\alpha_1(p), \alpha_2(p))$, define aliases 
$\alpha = \alpha_1$ and $\beta = \alpha_2$.

Thus, $\alpha(p) = \alpha_1(p)$ is the maximum length of a cup that ends with the vertex $p$ and an edge of label 1 ($\alpha(p) = 1$ if there is no such cup).
The value $\beta(p) = \alpha_2(p)$ is the maximum length of any cup that ends with the vertex $p$.
\end{definition}

The $\alpha$-statistic $\alpha(p) = (\alpha(p), \beta(p))$
maps $S$ to the triangular grid set 
$$T_{4, n} := \{(a, b) \in \NN^2 \colon 1 \leq a \leq b \leq n - 1 \}$$ 
injective by Theorem \ref{thm:inj}.
With this, if the size of $S$ is $|T_{4,n}| - k = \binom{n-1}{2} - k$ where $k$ is small, 
it helps to identify $S$ with the grid points $T_{4,n}$ with $k$ missing holes (see Figure \ref{fig:example-stat}).
Call such a diagram an \emph{$(\alpha, \beta)$-plane} of the 4-cap, $n$-cup free configuration $S$.
The following is a direct consequence of Theorem \ref{thm:inj},
and we will use it extensively without mentioning. 

\begin{corollary}
\label{lem:4capfree}
Assume an arbitrary 4-cap, $n$-cup free configuration $S$.

For any points $p, q$ with $\beta(p) = \beta(q)$,
\begin{itemize}
	\item $p$ and $q$ are always connected with an edge of label 1
	\item and $p < q$ if and only if $\alpha(p) < \alpha(q)$.
\end{itemize}
Consequently, in an $(\alpha, \beta)$-plane any horizontal edge is labeled 1, 
and each column is sorted in the increasing order of vertices from left to right.

Likewise, for any points $p, q$ with $\alpha(p) = \alpha(q)$,
\begin{enumerate}
	\item $p$ and $q$ are always connected with an edge of label 2
	\item and $p < q$ if and only if $\beta(p) < \beta(q)$.
\end{enumerate}
Consequently, in an $(\alpha, \beta)$-plane any vertical edges is labeled 2, 
and each row is sorted in the increasing order of vertices from bottom to top.
\end{corollary}

We will now use the concept of the $(\alpha, \beta)$-plane to introduce and motivate several key notions for our proof.
Take a 4-cap, $n$-cup free configuration $S$.
For simplicity, assume at first that $S$ of size $\binom{n}{2}$ (the maximum possible size), so that every location in the $(\alpha, \beta)$-plane is occupied.
Then we can identify a point $p$ in $S$ with its location $(\alpha(p), \beta(p))$ in the $(\alpha, \beta)$-plane.

Take any $1 \leq k \leq n-1$ and consider the path of vertices 
$$C_k = (1,k), (2,k), \dots, (k-1, k), (k,k), (k, k+1), \dots, (k, n-1)$$ 
in the $(\alpha, \beta)$-plane.
By Corollary \ref{lem:4capfree}, the first $k-1$ edges in this path are labeled 1, and the last $n-k-1$ edges are labeled 2, so these points must form an $(n-1)$-cup $C_k$ from start $p_k = (1, k)$ to end $q_k = (k, n - 1)$.
For any pair of $(n-1)$-cups $C_{k}$ and $C_{l}$ with $k < l$, the ordering of their endpoints is then $p_{k} < p_{l} \leq q_{k} < q_{l}$.
Motivated by this, we introduce the following definition of \emph{interweaved cups}.
\begin{definition}
\label{def:inter-laced}
Two cups $C_1$ and $C_2$ running from $p$ to $r$ and $q$ to $s$ respectively are \emph{interweaved} if $p < q \leq r < s$ holds.
\end{definition}
So if $|S| = \binom{n}{2}$, we get $n-1$ different $(n-1)$-cups $C_1, \dots, C_{n-1}$ all mutually interweaved.
Moreover, for each cup $C_k$ 
we have the $k$-cup $D_k$ of vertices 
$$D_k = (1,1), (1,2), ..., (1,k)$$
that ends with the left endpoint $p_k$ of $C_k$ and the $(n-k)$-cup $E_k$ of vertices 
$$E_k = (k, n-1), (k+1, n-1), ..., (n-1, n-1)$$
that starts with the right endpoint $q_k$ of $C_k$. 
Motivated by this, we introduce the following notion of \emph{laced cups}.
\begin{definition}
\label{def:laced}
A $(n-1)$-cup $C$ from $p$ to $q$ is \emph{laced} if there exists a cup $C_p$ that ends with $p$, and a cup $C_q$ that starts with $q$, so that $|C_p| + |C_q| = n - 1$.\footnote{One might expect the sum to be $n$ given the above motivation, but it turns out that $n-1$ is sufficient.}
\end{definition}
So if $|S| = \binom{n}{2} = \binom{n-1}{2} + (n - 1)$, then all of our mutually interweaved $(n-1)$-cups $C_k$ are laced.
The importance of this concept will be shown in Lemma \ref{lem:inter-gon} which shows that if $S$ contains just \emph{two} interweaved laced $(n-1)$-cups, then $S$ contains a $(3,n-1)$-gon.
We expect that for $|S| = \binom{n-1}{2} + d$ with any $1 \leq d \leq n - 1$, we can find $d$ mutually interweaved laced $(n-1)$-cups (Conjecture \ref{conj:interweaved}). 
Roughtly, this amounts to saying that an additional hole in the $(\alpha, \beta)$-plane only destroys one of the mutually interweaved laced $(n-1)$-cups.

We show this for $d = 2$ (Theorem \ref{thm:main-lemma}) and this is sufficient to show the main theorem that for $|S| = \binom{n-1}{2} + 2$ we can always find an $n$-gon (Theorem \ref{thm:main-thm}).
The proof for case $d=2$ requires a delicate inductive argument and several lemmas about interweaved laced $(n-1)$-cups; we now turn to stating and proving those lemmas.
Before doing so, we briefly remark that the concept of interweaved cups and of laced cups are symmetric under mirror reflection.
\begin{lemma}
\label{lem:laced-mirror}
Two cups $C_1, C_2$ are interweaved in configuration $S$ if and only if $C_2^\op$ and $C_1^\op$ are interweaved in $S^\op$.
An $(n-1)$-cup $C$ is laced in configuration $S$ 
if and only if $C^\op$ is laced in $S^\op$.
\end{lemma}

\section{Interweaved laced cups}
\label{sec:laced}

We first show that in a 4-cap free configuration, a pair of interweaved laced $(n-1)$-cups from $p$ to $r$ and $q$ to $s$ respectively (so that $p < q \leq r < s$) is sufficient to force an $n$-gon.
The following covers the degenerate case $q = r$. 

\begin{lemma}[Balko and Valtr \cite{balko2017sat}]
\label{lem:cupcup}
Take any $n \geq 3$ and a 4-cap, $n$-cup free configuration $S$.
If the ending point $x$ of an $(n-1)$-cup $C_1$
is also the starting point of another $(n-1)$-cup $C_2$, then $S$ contains an $(3, n-1)$-gon. 	
\end{lemma}
\begin{center}
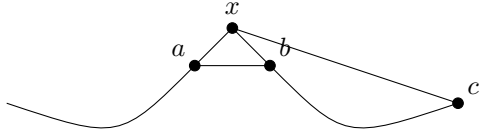
\begin{figure}
\begin{tikzpicture}
\draw (-3,-1) .. controls (-1.5,-1.5) .. (-0.5,-0.5) -- (0, 0);
\draw (0, 0) -- (0.5, -0.5) .. controls (1.5,-1.5) .. (3,-1);
\draw (-0.5,-0.5) -- (0.5, -0.5);
\draw (0, 0) -- (3, -1);

\draw[fill] (0, 0) circle (2pt) node[anchor=south, outer ysep=2pt]{$x$};
\draw[transparent] (0, -0.5) circle (2pt) node[anchor=north]{1};

\draw[fill] (-0.5, -0.5) circle (2pt) node[anchor=south east]{$a$};
\draw[fill] (0.5, -0.5) circle (2pt) node[anchor=south west]{$b$};
\draw[fill] (3, -1) circle (2pt) node[anchor=south west]{$c$};

\draw[transparent] (1.5, -0.7) circle (2pt) node[anchor=north]{$Q$};
\draw[transparent] (-1.5, -0.7) circle (2pt) node[anchor=north]{$P$};
\end{tikzpicture}
\caption{Figure for the proof of Lemma \ref{lem:cupcup}. The cup $P$ ends at $a$ and the cup $Q$
is from $b$ to $c$.}
\label{fig:lem-cupcup}
\end{figure}
\end{center}
\begin{proof}
(See Figure \ref{fig:lem-cupcup}) Say $C_1 = Px$ where $P$ is an $(n-2)$-cup that ends with some vertex $a$.
Likewise, say that $C_2 = xQ$ where $Q$ is an $(n-2)$-cup that starts with $b$.
Note that the statement to prove and the definitions introduced so far are symmetric under reflection.
So without loss of generality, we can assume that $ab$ is labeled 1.
Now $aQ$ is a $(n-1)$-cup because $ab$ has label 1.
Say that $Q$ ends with the point $c$.

If $axc$ is a cap, then we find the $(3, n-1)$-gon formed by $axc$ and $aC_2'$ and we are done.
If $axc$ is a cup, then the $(n-1)$-cup $C_1$ extends to the right with vertex $c$, contradicting that $S$ is $n$-cup free.
\end{proof}

Now we show the general case $q \leq r$.
\begin{lemma}
\label{lem:inter-gon}
Take any 4-cap, $n$-cup free configuration $S$ where $n \geq 3$.
If $S$ contains a pair of interweaved laced $(n-1)$-cups, then $S$ contains an $(3, n-1)$-gon.
\end{lemma}

\begin{center}
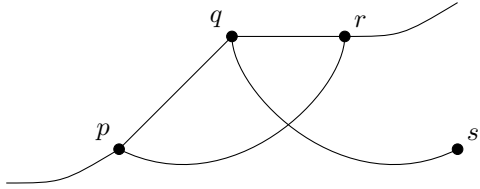
\begin{figure}
\begin{tikzpicture}[scale=1.5]
\draw (-2,-0.5) .. controls (-1, -1) and (0, 0) .. (0, 0.5);
\draw (1,-0.5) .. controls (0, -1) and (-1, 0) .. (-1, 0.5);
\draw (-1, 0.5) -- (0, 0.5);
\draw (-2, -0.5) -- (-1, 0.5);

\draw (-2, -0.5) .. controls (-2.5, -0.8) .. (-3, -0.8);
\draw (0, 0.5) .. controls (0.5, 0.5) .. (1, 0.8);

\draw[fill] (-2, -0.5) circle (1.3333pt) node[anchor=south east]{$p$};
\draw[fill] (-1, 0.5) circle (1.3333pt) node[anchor=south east]{$q$};
\draw[fill] (0, 0.5) circle (1.3333pt) node[anchor=south west]{$r$};
\draw[fill] (1, -0.5) circle (1.3333pt) node[anchor=south west]{$s$};
\draw[transparent] (-2.7, -0.6) circle (1.3333pt) node {$C_p$};
\draw[transparent] (0.7, 0.85) circle (1.3333pt) node {$C_r$};

\draw[transparent] (-0.5, 0.5) circle (1.3333pt) node[anchor=south] {1};
\draw[transparent] (-1.5, 0) circle (1.3333pt) node[anchor=south east] {2};

\draw[transparent] (-1.2, -0.4) circle (1.3333pt) node {$C_1$};
\draw[transparent] (0.2, -0.4) circle (1.3333pt) node {$C_2$};

\end{tikzpicture}
\caption{Figure for the proof of Lemma \ref{lem:inter-gon}. The cup $C_p$ ends with $p$ and $C_r$ starts with $r$.}
\label{fig:lem-inter-gon}
\end{figure}
\end{center}

\begin{proof}
(See Figure \ref{fig:lem-inter-gon}) Let $C_1$ and $C_2$ be the pair of interweaved laced $(n-1)$-cups from $p$ to $r$ and $q$ to $s$ respectively, so that $p < q \leq r < s$.
If $q = r$, we are done by Lemma \ref{lem:cupcup}, so we can assume $q < r$.
As the setup is symmetric along reflection (e.g. Remark \ref{rem:4cap-mirror} and \ref{lem:laced-mirror}), we can assume that the edge $qr$ is of label 1 without loss of generality.

As $C_1$ is laced we have a cup $C_p$ ending with $p$ and
and a cup $C_r$ starting with $r$ so that $|C_p| + |C_r| = n - 1$.
Observe that the edge $pq$ is of label 2, or otherwise we can extend the $(n-1)$-cup $C_2$ to left with $pq$
and reach contradiction.

If $pqr$ is a cap, then the cap $pqr$ and the cup $C_1$ forms a $(3, n-1)$-gon and we are done.
Now assume that $pqr$ is a cup. 
Since $pq$ is of label 2, the cup $C_p$ extends to right with $q$.
Since $qr$ is of label 1, the cup $C_r$ extends to left with $q$ as well.
Now the cups $C_pq$ and $qC_r$ are joined along the vertex $q$.
Since $pqr$ is a cup, they all connect to make a cup $C_pqC_r$ of size $|C_p| + |C_r| + 1 = n$, leading to contradiction.
\end{proof}

We prepare more lemmas and observations before starting the main proof.
In particular, we will use the following special case to find a pair of interweaved laced cups.
Note that for this lemma we have $n \geq 4$ instead of $n \geq 3$.
\begin{lemma}
\label{lem:special}
Take any 4-cap, $n$-cup free configuration $S$ with $n \geq 4$.
Assume an $(n-1)$-cup $C$ from $x$ to $y$ in $S$.
Also, assume an $(n-2)$-cup $C_x$ that ends with $x$, and an $(n-2)$-cup $C_y$ that starts with $y$.
Then $S$ contains a pair of interweaved laced $(n-1)$-cups.
\end{lemma}

\begin{center}
\begin{figure}
\begin{tikzpicture}
\draw (-3,-1) .. controls (-1.5,-1.5) .. (-0.5,-0.5) -- (0, 0);
\draw (0, 0) -- (0.5, -0.5) .. controls (2, -2) and (4, -2) .. (5.5, -0.5) -- (6, 0);
\draw (6, 0) -- (6.5, -0.5) .. controls (7.5,-1.5) .. (6+3,-1);
\draw[fill] (0, 0) circle (2pt) node[anchor=south]{$x$};
\draw[fill] (-0.5, -0.5) circle (2pt) node[anchor=south]{$a$};
\draw[fill] (0.5, -0.5) circle (2pt) node[anchor=south]{$b$};
\draw[fill] (6, 0) circle (2pt) node[anchor=south]{$y$};
\draw[fill] (6.5, -0.5) circle (2pt) node[anchor=south]{$d$};
\draw[fill] (5.5, -0.5) circle (2pt) node[anchor=south]{$c$};
\draw[transparent] (3, -1.5) circle (2pt) node[anchor=south]{$Q$};
\draw[transparent] (-1.5, -0.5) circle (2pt) node[anchor=north]{$P$};
\draw[transparent] (7.5, -0.5) circle (2pt) node[anchor=north]{$R$};
\end{tikzpicture}
\caption{Figure for the proof of Lemma \ref{lem:special}. The cups $P$, $Q$ and $R$
contains the endpoints $a$, $b$ and $c$, and $d$ respectively.}
\label{fig:lem-special}
\end{figure}
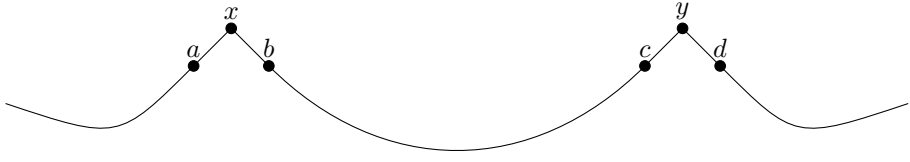
\end{center}

\begin{proof}
Using the symmetry of the statement along mirror reflection,
we can assume that $xy$ is labeled 1 without loss of generality.

Let $C_x = Px$, $C = xQy$ and $C_y=yR$ so that $P, Q, R$ are $(n-3)$-cups (see Figure \ref{fig:lem-special}).
Let $a$ be the ending point of $P$ and $d$ be the starting point of $R$.
Say that $Q$ is from point $b$ to $c$. 
As $|P|, |Q|, |R| \geq 1$, the points $a, b, c$ and $d$ are all well-defined and we have $a < x < b \leq c < y < d$ (there is a possibility for $b=c$ when $n = 4$).

The sequence $xyR$ is an $(n-1)$-cup because $xy$ is labeled 1 and $C_y = yR$ is $(n-2)$-cup. 
It is also laced because $x$ is the endpoint of the $(n-2)$-cup $C_x = Px$.
So we have two laced $(n-1)$-cups $xQy = C$ and $xyR$.
 
Do case analysis on the label of $ab$. Assume first that $ab$ is labeled 1. 
Then $aQy$ is an $(n-1)$-cup. It is also laced because $a$ is a 1-cup and $yR$ is a $(n-2)$-cup. 
$aQy$ interweaves with the laced $(n-1)$-cup $xyR$ because $a < x < y$ and the endpoint of $R$ is strictly on the right of $y$.

Assume now that $ab$ is labeled 2. 
We do another case analysis on the label of $by$.
If $by$ is labeled 1, then $byR$ is a laced $(n-1)$-cup because it is a $(n-1)$-cup and $Pb$ is a $(n-2)$-cup.
Now the laced $(n-1)$-cups $xQy$ and $byR$ are interweaved because $x < b < y$ and the endpoint of $R$ is strictly on the right of $y$.
If $by$ is labeled 2, then $Pby$ is an $(n-1)$-cup, and it is also laced because $yR$ is an $(n-2)$-cup.
Now laced $(n-1)$-cups $Pby$ and $xyR$ are interweaved, because the starting point of $P$ comes before $x$, $x < y$, and the endpoint of $R$ is strictly on the right side of $y$.
The proof is now done.
\end{proof}

We introduce the following terminologies for a 4-cap, $n$-cup free configuration $S$ of size at least $\binom{n-1}{2} + 1$.
Note that the following definition does not depend on a particular choice of a slope labeling of $S$.
\begin{definition}
\label{def:endpts}
Assume an arbitrary 4-cap, $n$-cup free configuration $S$ of size at least $\binom{n-1}{2} + 1$ with $n \geq 3$.

Define $L(S)$ as the set of left endpoints of all $(n-1)$-cups in $S$, and $R(S)$ as the set of right endpoints of all $(n-1)$-cups in $S$.
Let $p_S$ be the rightmost point of $L(S)$ and $q_S$ be the leftmost point of $R(S)$.
By Theorem \ref{thm:capcup-set}, there is at least one $(n-1)$-cup in $S$, so that $L(S), R(S)$ are nonempty and the points $p_S, q_S$ are well-defined.
\end{definition}

\begin{remark}
\label{rem:psqs}
Note that under the assumption of Definition \ref{def:endpts},
$L(S^\op) = R(S)$, $R(S^\op) = L(S)$, $p_{S^\op} = q_S$ and $q_{S^\op} = p_S$.
\end{remark}

\begin{lemma}
\label{lem:endpts}
Under the assumption of Definition \ref{def:endpts}, 
we always have $p_S \leq q_S$.
In other words, the sets $L(S)$ and $R(S)$ are separated in the increasing order with at most one overlap.
\end{lemma}
\begin{proof}
Assume the contrary that $p_S > q_S$. 
Take any slope labeing of $S$. If the edge $q_Sp_S$ is of label 1 we can extend an $(n-1)$-cup starting with $p_S$ to the left by $q_S$, finding an $n$-cup in $S$.
Likewise, if the edge $q_Sp_S$ is of label 2 we can extend an $(n-1)$-cup starting with $q_S$ to the right by $p_S$, finding an $n$-cup in $S$.
Both cases lead to contradiction.
\end{proof}

We rule out the following case where we can find a pair of interweaved laced $(n-1)$-cups immediately as well.
\begin{lemma}
\label{lem:david}
Assume an arbitrary 4-cap, $n$-cup free configuration $S$ of size at least $\binom{n-1}{2} + 1$.
If a $(n-2)$-cup in $S$ ends with $p_S$ and a $(n-2)$-cup in $S$ starts with $q_S$,
then $S$ contains a pair of interweaved laced $(n-1)$-cups.
\end{lemma}
\begin{proof}
An $(n-1)$-cup $C_l$ from some $p' \in L(S)$ to $q_S$ exists by the definition of $q_S$ and $L(S)$.
Likewise, an $(n-1)$-cup $C_r$ from $p_S$ to some $q' \in R(S)$ exists by the definition of $p_S$ and $R(S)$.
Note that $C_l$ (or $C_r$) is laced by the existence of a $(n-2)$-cup that ends with $p_S$ (or that starts with $q_S$).

We have $p' \leq p_S$ by the definition of $p_S$, $p_S \leq q_S$ by Lemma \ref{lem:endpts} and  $q_S \leq q'$ by the definition of $q_S$.
If $p' = p_S$, then we can apply lemma \ref{lem:special} to $C_l$ to conclude the proof.
If $q_S = q'$, then we can apply lemma \ref{lem:special} to $C_r$ to conclude the proof.
If none of such equalities hold, then $p' < p_S \leq q_S < q'$ holds so $C_l$ and $C_r$ forms a pair of interweaved laced $(n-1)$-cups.
\end{proof}

We define the rows of $S$ in the $(\alpha, \beta)$-plane and show that each row is nonempty.
\begin{definition}
\label{def:rows}
Fix an arbitrary 4-cap, $n$-cup free configuration $S$ of size at least $\binom{n-1}{2} + 1$ with $n \geq 3$ and a fixed slope labeing and $\alpha$-statistic.
For any $1 \leq i \leq n - 1$, define 
$$R_i(S) = \{p \in S : \beta(p) = i\} $$
which is the $i$'th row of the $(\alpha, \beta)$-plane of $S$ from the bottom.
Note that $R_{n-1}(S) = R(S)$ in particular by definition.
\end{definition}
\begin{lemma}
\label{lem:row-nonempty}
Under the assumption of Definition \ref{def:rows}, 
each row $R_i(S)$ is nonempty for all $1 \leq i \leq n - 1$.
\end{lemma}
\begin{proof}
Define the set $R_{\geq i}(S) = \{p \in S : \beta(p) \geq i\}$ for all $1 \leq i \leq n - 1$.
By Theorem \ref{thm:capcup-set}, there is at least one $(n-1)$-cup $C$ in $S$, so that the right endpoint of $C$ is contained in $R_{\geq i}(S)$ for all $i$.
Therefore, the minimum (leftmost) point $x_i$ of $R_{\geq i}(S)$ exists for all $i$. 

We show that $x_i < x_{i+1}$ for all $1 \leq i < n - 1$.
As $x_{i+1}$ is in $R_{\geq(i+1)}(S)$, we can take a
$(i+1)$-cup $Cx_{i+1}$ that ends with $x_{i+1}$.
Let $y$ be the ending point of the $i$-cup $C$. 
Then $y \in R_{\geq i}(S)$ by definition and also $y < x_{i+1}$ because $Cx_{i+1}$ forms a cup.
This implies that $x_i = \min R_{\geq i}(S) \leq y < x_{i+1}$.

Because $S$ is $n$-cup free, the equality $R_{n-1}(S) = R_{\geq (n-1)}(S)$ holds and the set is nonempty.
For all $1 \leq i < n - 1$, the set $R_{i}(S) = R_{\geq (i + 1)}(S) \setminus R_{\geq i}(S)$ contains $x_i$ because $x_i < x_{i+1}$.
This concludes the proof.
\end{proof}

Now we are ready to prove the existence of a pair of interweaved laced $(n-1)$-cups by induction.
\begin{theorem}
\label{thm:main-lemma}
For any $n \geq 3$, any 4-cap, $n$-cup free configuration of size $\binom{n-1}{2}+2$ contains a pair of interweaved laced $(n-1)$-cups.
\end{theorem}

\begin{proof}
We proceed by induction.
The base case $n = 3$ can be checked as the following. For any configuration of size $\binom{2}{2} + 2 = 3$, if the points are $x, y, z$ in the increasing order, then $xy$ and $yz$ forms a pair of interweaved laced $2$-cups by definition. 

Now we show the inductive step. 
Fix an arbitrary $n \geq 4$ and the following inductive hypothesis ($\ast$).
\begin{itemize}
\item[($\ast$)] Any 4-cap, $(n-1)$-cup free configuration of size $\binom{n-2}{2} + 2$ contains a pair of interweaved laced $(n-2)$-cups.
\end{itemize}
Fix an arbitrary 4-cap, $n$-cup free configuration $S$ of size $\binom{n-1}{2}+2$. 
Our goal now is to show that $S$ contains a pair of interweaved laced $(n-1)$-cups.

By Lemma \ref{lem:david}, 
we are already done if $p_S$ is an endpoint of a $(n-2)$-cup
and $q_S$ is the starting point of an $(n-2)$-cup at the same time.
It remains for us to show the case where either $p_S$ is not the endpoint of a $(n-2)$-cup, or $q_S$ is not the starting point of an $(n-2)$-cup.
Without loss of generality (e.g. Remark \ref{rem:psqs}), we can assume that $q_S$ is not the starting point of an $(n-2)$-cup by reflecting $S$ in the other case.
Now fix a specified slope labeling and $\alpha$-statistic of $S$ to use (Theorem \ref{thm:label}).

\begin{center}
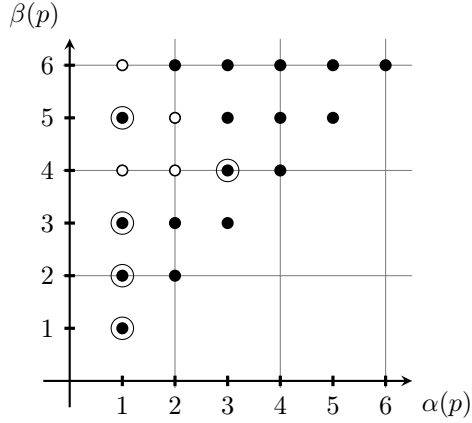
\begin{figure}[h]
\centering 
\begin{tikzpicture}[scale=0.7]

\draw[step=2cm,gray] (-0.5,-0.5) grid (6.5,6.5);
\draw[-stealth,thick] (-0.5,0)--(6.5,0) node[below right]{$\alpha(p)$}; 
\draw[-stealth,thick] (0,-0.5)--(0,6.5) node[above left]{$\beta(p)$};
\foreach \i in {1,2,...,6}
\draw[very thick] 
(\i,.1)--(\i,-.1) node[below,fill=white]{\i}
(-.1,\i) node[left,fill=white]{\i}--(.1,\i) ;

\begin{scope}[every node/.style={circle,thick,draw,fill=black,scale=0.4}]
\foreach \i in {1,...,6}
\foreach \j in {1,...,\i}
\node at (\j, \i) {};
\end{scope}

\begin{scope}[every node/.style={circle,draw,fill=white,scale=0.4}]
\node at (1, 6) {};
\node at (2, 5) {};
\node at (1, 4) {};
\node at (2, 4) {};
\end{scope}

\begin{scope}[every node/.style={circle,draw,fill=none,scale=0.9}]
\node at (1, 1) {};
\node at (1, 2) {};
\node at (1, 3) {};
\node at (3, 4) {};
\node at (1, 5) {};
\end{scope}
\begin{scope}[every node/.style={circle,transparent,anchor=south west,fill=none,scale=0.8}]
\node at (1, 1) {$x_1$};
\node at (1, 2) {$x_2$};
\node at (1, 3) {$x_3$};
\node at (3, 4) {$x_4$};
\node at (1, 5) {$x_5$};
\node at (2, 6) {$x_6$};
\end{scope}

\end{tikzpicture}
\caption{A hypothetical $(\alpha, \beta)$-plane of a 4-cap, $n$-cup free set $S$ of size $\binom{n-1}{2} + 2$ with $n = 7$. There are $n - 3 = 4$ `holes' not in $S$ drawn as white points. The set $\Delta$ of size $n - 2 = 5$ is marked with circles.}
\label{fig:sprime}
\end{figure}
\end{center}

We define the set $S'$ on which we will apply the inductive hypothesis ($\ast$) as the following. 
For each $1 \leq i \leq n - 1$, define $x_i$ as the leftmost point of the $i$'th row $R_i(S)$ of the $(\alpha, \beta)$-plane (Definition \ref{def:rows}).
By Lemma \ref{lem:row-nonempty}, each $x_i$ is well-defined.
Define the set $\Delta = \{x_1, x_2, \cdots, x_{n-2}\}$ of size $n-2$.
Note that the point $x_{n-1}$ is excluded from $\Delta$ by definition,
and that $x_{n-1} = q_S$ by Definition \ref{def:endpts}.
Define the set $S'$ of size $\binom{n-2}{2} + 2$ as $S' = S \setminus \Delta$ (see Figure \ref{fig:sprime}).
We will show later that $S'$ is $(n-1)$-cup free, so that we can apply the induction hypothesis ($\ast$) to $S'$.

We assumed without loss of generality that $q_S$ is not the starting point of an $(n-2)$-cup. Using it, we show the following consequences.
\begin{itemize}
\item[($\ast\ast$)] Let $x$ be the starting point of an arbitrary $(n-2)$-cup $C$ in $S'$ and let $i = \beta(x)$. Then the following hold.
\begin{enumerate}
	\item $x < q_S$
	\item $x \not\in R_{n-1}(S)$ and $i \leq n - 2$
	\item $x_i < x$ and the edge $x_ix$ is labeled 1, so $x_iC$ is an $(n-1)$-cup in $S$
	\item $\alpha(x_i) = 1$ and $\alpha(x) = 2$
\end{enumerate}
\end{itemize}

First we show $x < q_S$, the first item of ($\ast\ast$). Assume otherwise that $x \geq q_S$. Then the case $x = q_S$ directly contradicts our assumption, so we should have $x > q_S$.
If the edge $q_Sx$ is labeled 1, then $q_S C$ is an $(n-1)$-cup so it also contradicts our assumption.
So $q_Sx$ should be labeled 2. But in this case,
as $q_S$ is the ending point of some $(n-1)$-cup $D$ by definition, $Dx$ is an $n$-cup which contradicts that $S$ is $n$-cup free.
This concludes the proof of $x < q_S$ by contradiction.

As $q_S$ is the minimum value of $R(S) = R_{n-1}(S)$, we also get $x \not\in R_{n-1}(S)$ and $i = \beta(x) \leq n - 2$ (the second item of ($\ast\ast$)).
Therefore, $x_i \in \Delta$ and as $x_i$ is minimal in the row $R_i(S)$ that also contains $x$,
we have $x_i < x$ and the edge $x_i x$ is labeled 1.
Consequently, the $(n-1)$-cup $C$ extends to left with $x_i$ (the third item of ($\ast\ast$)).

If $\alpha(x_i) > 1$, then we can extend the $(n-1)$-cup $x_iC$ to left in $S$, so it contradicts that $S$ is $n$-cup free.
So $\alpha(x_i) = 1$.
Next, we show that $\alpha(x) = 2$.
Because the edge $x_{i}p$ is labeled 1,
we have $\alpha(p) \geq 2$.
If $\alpha(p) > 2$, then we have a 3-cup $D$ that ends with $p_\epsilon$ and an edge with label 1.
Now the cup $D$ and $C_\epsilon$ meets at $p_\epsilon$, and they form a cup because $D$ ends with an edge with label 1. 
The cup is of size $|D| + |C_\epsilon| - 1 = n$ and we reach contradiction.
So the equality $\alpha(p_\epsilon) = 2$ holds (the fourth item of ($\ast\ast$)).
This concludes the proof of ($\ast\ast$).

We now show that $S'$ is $(n-1)$-cup free.
Assume otherwise that $S'$ contains an $(n-1)$-cup $C_0$ from some point $x$ to $y$.
Let $i = \beta(x)$. 
Then since $x$ is also the starting point of some $(n-2)$-cup in $S'$, the property ($\ast\ast$) implies that  the edge $x_i x$ has label 1.
Now the $(n-1)$-cup $C_0$ in $S'$ extends to left with $x_i$ in $S$, contradicting that $S$ is $n$-cup free.

As $S'$ is $(n-1)$-cup free, we can apply the inductive hypothesis ($\ast$).
Doing so, we find a pair $(C_1, C_2)$ of interweaved laced $(n-2)$-cups in $S'$, each from $p_1$ to $r_1$ and $p_2$ to $r_2$ so that $p_1 < p_2 \leq r_1 < r_2$.
Define $i_1 = \beta(p_1)$ and $i_2 = \beta(p_2)$. 
By the property ($\ast\ast$) applied to $C_1$ and $C_2$ respectively, we have the following.
\begin{itemize}
	\item $p_1, p_2 < q_S$
	\item $i_1, i_2 < n - 1$
	\item $x_{i_1} p_1$ and $x_{i_2}p_2$ are edges with label 1
	\item $x_{i_1}C_1$ are $x_{i_2}C_2$ are $(n-1)$-cups in $S$
	\item $\alpha(x_{i_1}) = \alpha(x_{i_2}) = 1$ and $\alpha(p_1) = \alpha(p_2) = 2$
\end{itemize}
We show that the pair $(x_{i_1}C_1, x_{i_2}C_2)$ of $(n-1)$-cups is interweaved and laced.
This completes the inductive step and the whole proof.

First, we show that the pair is interweaved.
That is, whether $x_{i_1} < x_{i_2} \leq r_1 < r_2$ is true.
By the assumptions $x_{i_2} < p_2 \leq r_1 < r_2$,
it only remains for us to show that $x_{i_1} < x_{i_2}$.
Applying Corollary \ref{lem:4capfree} to $p_1 < p_2$ and $\alpha(p_1) = \alpha(p_2) = 2$, we have $\beta(p_1) < \beta(p_2)$ which is exactly $i_1 < i_2$ by definition.
Applying Corollary \ref{lem:4capfree} again to $\beta(x_{i_1}) = i_1 < i_2 = \beta(x_{i_2})$ and $\alpha(x_{i_1}) = \alpha(x_{i_2}) = 1$,
we have $x_{i_1} < x_{i_2}$.
This completes the proof that the pair $(x_{i_1}C_1, x_{i_2}C_2)$  is interweaved.

Now we show that each cup in the pair $(x_{i_1}C_1, x_{i_2}C_2)$ is laced.
Let $\epsilon \in \{1, 2\}$ be any of the index 1 or 2.
We extend the laced $(n-2)$-cup $C_\epsilon$ in $S'$ to show that $x_{i_\epsilon}C_\epsilon$ in $S$ is also laced.
Because $C_\epsilon$ from $p_\epsilon$ to $r_\epsilon$ is laced in $S'$, there are cups $C_{p_\epsilon}$ and $C_{r_\epsilon}$ in $S'$
that ends with $p_\epsilon$ and starts with $r_\epsilon$ respectively,
so that $|C_{p_\epsilon}| + |C_{r_\epsilon}| = n - 2$.

Say that $C_{p_\epsilon}$ starts with a point $z$ in $S'$.
Then $z \leq p_\epsilon < q_S$ so $\beta(z) < n - 1$ by the definition of $q_S$.
Consequently, we have $\alpha(z) \geq 2$ as $\beta(z) < n - 1$ and $z$ is not in $\Delta$.
So $z$ is the endpoint of an edge $yz$ of label 1 in $S$.
We can extend $C_{p_\epsilon}$ to left to a cup $yC_{p_\epsilon}$ of size $|C_{p_\epsilon}| + 1$ in $S$, and the cup $yC_{p_\epsilon}$ ends with $p_\epsilon$.
This implies that $|C_{p_\epsilon}| + 1 \leq \beta(p_\epsilon) = i_\epsilon = \beta(x_{i_\epsilon})$.
So in particular, $\beta(x_{i_\epsilon}) + |C_{r_i}| \geq |C_{p_\epsilon}| + |C_r| + 1 = n - 1$
and $x_{i_\epsilon}C_\epsilon$ is a laced $(n-1)$-cup for any of $\epsilon = $ 1 or 2. This concludes the proof.

\end{proof}

We obtain the main theorem (Theorem \ref{thm:main-thm}) immediately as the following.
\begin{proof}[Proof of Theorem \ref{thm:main-thm}]
Combine Theorem \ref{thm:main-lemma} with Lemma \ref{lem:inter-gon}.
\end{proof}

We end with a conjecture that generalizes Theorem \ref{thm:main-lemma}.
\begin{conjecture}
\label{conj:interweaved}
For any $n \geq 3$ and $1 \leq k \leq n$, any 4-cap, $n$-cup free configuration of size $\binom{n-1}{2}+k$ contains $k$ mutually interweaved laced $(n-1)$-cups.
\end{conjecture}
Theorem \ref{thm:main-lemma} is a special case $k=2$ of this conjecture. We can also prove the case $k=1$ with a similar induction argument with the same $S' = S \setminus \Delta$.
The case $k = n$ is an immediate consequence of the discussion at the end of Section \ref{sec:alpha-beta}.
The proof of Theorem \ref{thm:main-lemma} fails to extend because it exploits the mirror symmetry and essentially applies the inductive hypothesis twice to force the 'rightmost' laced $(n-1)$-cup starting with $p_S$ and the 'leftmost' laced $(n-1)$-cup ending with $q_S$.
This special case is then covered by Lemma \ref{lem:special}.

\section*{Acknowledgement}
The author is indebted to David Speyer for many illuminating discussions and his help in making the presentation much more organized and clear.
The author thanks to Andreas Holmsen for directing the author to the work of Moshkovitz and Shapira \cite{moshkovitz2014ramsey} that inspired the notion of $(\alpha, \beta)$-plane. 
The author also thanks to Martin Balko and Pavel Valtr for their work \cite{balko2017sat} that inspired the main Theorem \ref{thm:main} and their thoughtful feedback on the early version of the draft.
\bibliographystyle{abbrv}
\bibliography{main.bib}

\end{document}